\documentclass{amsart}
\usepackage{hyperref,enumerate}

\newcommand{\tN}{\widetilde{N}}
\newcommand{\Ri}{\widetilde{R}_i}

\newcommand{\hs}{\hskip.02in}
\newcommand{\wtg}{\widetilde{g}}
\newcommand{\Ud}{U_{\delta}}
\newcommand{\di}{dV_{\widetilde{g}_i}}
\newcommand{\wtd}{\widetilde{g_{\delta}}}
\newcommand{\wti}{\widetilde{g}_i}

\newtheorem{theorem}{Theorem}[section]

\newtheorem{proposition}[theorem]{Proposition}
\newtheorem{corollary}[theorem]{Corollary}

\theoremstyle{definition}

\theoremstyle{remark}
\newtheorem{remark}[theorem]{Remark}
\numberwithin{equation}{section}

\title{Poincar\'e-Einstein metrics and Yamabe invariants}

\author{Matthew J. Gursky}
\address{Department of Mathematics
         University of Notre Dame\\
         Notre Dame, IN 46556}
\email{\href{mgursky@nd.edu}{mgursky@nd.edu}}

\author{Qing Han}
\address{Department of Mathematics\\
         University of Notre Dame\\
         Notre Dame, IN 46556}
\email{\href{qhan@nd.edu}{qhan@nd.edu}}
\address{Beijing International Center for Mathematical Research\\
         Peking University\\
         Beijing, 100871, China}
\email{\href{qhan@math.pku.edu.cn}{qhan@math.pku.edu.cn}}

\begin{document}

\date{\today}

\begin{abstract}
In this note we prove the existence of infinitely many positive conformal
classes on $S^7$ which cannot be the conformal infinity of a Poincar\'e-Einstein metric on the ball $B^8$.  We also
prove a sharp inequality between the Yamabe invariant of the conformal infinity and the Yamabe invariant of the interior
(after a suitable compactification).
\end{abstract}

\thanks{The first author acknowledges the support of NSF grant DMS-1509633.  The second author acknowledges the support of NSF
Grant DMS-1404596.  }

\maketitle


\section{Introduction} \label{Intro}

In this paper, we assume $(X, g_{+})$ is an $n$-dimensional Poincar\'e-Einstein manifold.  More precisely, $X$ is the interior of a compact manifold $\overline{X}$ with
boundary $\partial X = M$, and there is a defining function $\rho \in C^{\infty}(X)$ with $\rho > 0$ and $d\rho \neq 0$ on $\partial X$, and $\rho^2 g_{+}$ extends to a metric
$\overline{g}$ on the compact manifold with boundary $(\overline{X},\partial X)$.  Also, the metric $g_{+}$ satisfies the Einstein condition with negative Einstein constant, which we normalize so that
\begin{align} \label{PEdef}
Ric_{g_{+}} = - (n-1) g_{+}.
\end{align}
We will assume throughout that the compactified metric $\overline{g}$ is at least $C^2$ up to the boundary.  This compactification
defines a conformal class of metrics on the boundary $[ \gamma ]$, where $\gamma = \overline{g} \big|_{M}$ is called the {\em conformal infinity} of $(X,g_{+})$. The basic example is the Poincar\'e model for hyperbolic space on the unit ball $B^n \subset \mathbb{R}^n$.  In this case the conformal infinity is the standard conformal structure on the round sphere $S^{n-1}$.

Conversely, given a conformal class of metrics on the boundary $M = \partial X$ one can ask whether the interior admits a Poincar\'e-Einstein metric whose conformal infinity is the given conformal class.  Although there is no general existence theory for this problem, a seminal result was proved by Graham-Lee \cite{GL}: Given a metric $\gamma$ sufficiently close to the round metric $\gamma_0$ on the sphere $S^{n-1}$, there is a Poincar\'e-Einstein metric $g_{+}$ on the ball $B^{n}$ whose conformal infinity is $[\gamma]$. In \cite{WittenHolo}, Witten remarks that ``...one might ask what is the significance of the fact that the
Graham-Lee theorem presumably fails for conformal structures that are sufficiently far from the round one'' (see page 263).  Although Witten's interest is mainly in a Yang-Mills analog, he points out the consequences of the fact that restricting to small neighborhood of the round metric implies that the scalar curvature is positive.  He also remarks that the Graham-Lee result is likely optimal; i.e., ``the restriction to conformal structures that are sufficiently close to the standard one is probably also necessary for most values of (the dimension)'' and suggests a possible approach to finding a counterexample by ``...using a family of $S^d$'s that cannot be extended to a family of $\overline{B}_{d+1}$'s'' (see the footnote at the bottom of page 260 in \cite{WittenHolo} for more details in the Yang-Mills context).

Our main result in this note to prove the existence of infinitely many conformal classes on the seven-dimensional sphere $S^7$ which cannot be the conformal infinity of a Poincar\'e-Einstein metric on the ball $B^8$, thus confirming Witten's intuition in this dimension.  We rely on the construction of Gromov-Lawson \cite{GL}, which they used to prove that the space $\mathfrak{R}^{+}(S^7)$ of positive scalar curvature metrics on $S^7$ has infinitely many connected components (see Section \ref{CorSec} for a summary).   The precise statement is:

\begin{theorem} \label{MainIntro} There are infinitely many components of $\mathfrak{R}^{+}(S^7)$ containing metrics whose conformal class cannot be the conformal infinity of a Poincar\'e-Einstein metric on the eight-dimensional ball $B^8$.
\end{theorem}

The construction of Gromov-Lawson can be extended to dimensions $4k -1$, for all $k \geq 2$, and thus we expect Theorem \ref{MainIntro} to hold in these cases as well (see the Remark at the end of Section \ref{CorSec}).  However, we only provide a detailed proof for dimension seven.

Our proof also relies on the fact, first observed by J. Qing \cite{JQ}, that if the Yamabe invariant of the conformal infinity of a Poincar\'e-Einstein manifold is positive, then the Yamabe invariant (suitably defined) of the compactified manifold is positive (see Section \ref{CorSec}).   In Section \ref{SecMainThm} we give a different proof of this result by appealing to the work of Escobar \cite{EscobarJDG} and Brendle-Chen \cite{BC} on the Yamabe problem for manifolds with boundary.  Using the Yamabe metric, we prove a sharp inequality between the Yamabe invariant of the boundary and the Yamabe invariant of the compactified manifold with boundary.  An asymptotic expansion for solutions of singular Yamabe equations plays a crucial role in the derivation of this inequality.

A brief note about the organization and conventions of this paper.  In Section \ref{SecYM} we define the version of the Yamabe problem on manifolds with boundary we will use, and give a summary of the known existence results.  In Section \ref{SecMainThm} we prove the aforementioned inequality for Yamabe invariants, Theorem \ref{MainThm1}.  In the final section we prove Theorem \ref{MainIntro}.

Since we will be using various facts about the Yamabe problem, our notational convention will differ from most papers on the subject of Poincar\'e-Einstein manifolds: that is, our P-E manifold will be $n$-dimensional, $n \geq 3$, and the boundary will have dimension $n-1$.

\smallskip

\noindent {\bf Acknowledgements.}  It is a pleasure to thank Stephan Stolz, who was an invaluable resource for references on the topology of manifolds with PSC, in particular the construction of \cite{GL}.  We would also like to thank Robin Graham, who pointed out Witten's suggestion for constructing counterexamples in \cite{WittenHolo} after we had sent him a preliminary version of the paper.

\section{A compactification of Poincar\'e-Einstien metrics via the Yamabe problem}  \label{SecYM}

There are various constructions of defining functions which appear in the literature of P-E metrics, and they can be viewed a kind of `gauge choice'.  In the following, we want to compactify $(X,g_{+})$ to obtain a metric of
constant scalar curvature such that the boundary $\partial X = M$ is minimal. The existence of such a compactification is equivalent to solving (one version of) the {\em boundary Yamabe problem}, and was first studied by Escobar in \cite{EscobarJDG}, and subsequently by Brendle-Chen \cite{BC}.   We now provide a brief summary of these results.

Given a compact Riemannian manifold with boundary $(X, \partial X, \overline{g})$, define the functional 
\begin{align*} 
\mathcal{Y}[u] = \dfrac{ \int_{X} \big( \frac{4(n-1)}{(n-2)} |\nabla_{\overline{g}} u|^2 + R_{\overline{g}} u^2 \big)dV_{\overline{g}} + 2 (n-1) \oint_{\partial X} H_{\overline{g}} u^2 dS_{\overline{g}} }{\Big( \int_{X} u^{2n/(n-2)} dV_{\overline{g}} \Big)^{(n-2)/n}},
\end{align*}
where $R_{\overline{g}}$ is the scalar curvature and $H_{\overline{g}}$ the mean curvature of the boundary with respect to $\overline{g}$.  Let
\begin{align} \label{YXMdef}
Y(X,\partial X, [\overline{g}]) = \inf_{\substack{u \in C^{\infty}(\overline{X}) \\ u > 0}} \mathcal{Y}[u].   
\end{align}
In \cite{EscobarJDG}, Escobar showed that
\begin{align} \label{Aubin}
Y(X,\partial X,[\overline{g}]) \leq Y(S^n_{+}, \partial S^n_{+}, [g_0]) = n(n-1)(\frac{1}{2}\omega_n)^{2/n},
\end{align}
where $S^n_{+}$ is the upper hemisphere, $g_0$ is the round metric, and $\omega_n$ is the volume of $S^n$.  Moreover, when the inequality is strict then $Y(X,\partial X, [\overline{g}])$ is attained by a smooth function $u > 0$ which defines a conformal metric $g_Y = u^{\frac{4}{n-2}}\overline{g}$ with constant scalar curvature and minimal boundary:
\begin{align} \label{bdyYamabe} \begin{split}
R_{g_Y} &= Y(X,\partial X,[\overline{g}])\cdot  Vol(X,g_Y)^{-2/n} \ \ in \ X, \\
H_{g_Y} &= 0\ on \ \ \partial X.
\end{split}
\end{align}

Escobar was able to verify that the inequality in (\ref{Aubin}) was strict if $(X,\partial X, \overline{g})$ is not conformally equivalent to the hemisphere and the dimension $3 \leq n \leq 5$, or in dimensions $n \geq 6$ if $\partial X$ is not umbilic. In \cite{BC}, Brendle-Chen considered the remaining cases; i.e., $n \geq 6$ umbilic boundary.  Their work is particularly relevant to our setting since the Einstein condition of $g_{+}$ implies that its compactification $(\overline{X},\partial X, \overline{g})$ has totally umbilic boundary (see \cite{AndersonAM}, p. 210).

Brendle-Chen were able to verify the remaining cases subject to the validity of the Positive Mass Theorem (PMT).   More specifically, they considered the conformal metric defined by $h = G^{4/(n-2)}\overline{g}$, where $G > 0$ is the Green's function for the conformal laplacian with Neumann boundary conditions and pole $p \in \partial X$.  If $\widetilde{X} = X \cup_{\partial X} (- X)$ denotes the double of $X$, then $h$ can be extended to a metric $\widetilde{h}$ on $\widetilde{X} \setminus \{ p\}$ so that $(\widetilde{X} \setminus \{ p\}, \widetilde{h})$ becomes an asymptotically flat manifold with zero scalar curvature (see \cite{BC}, Proposition 4.3).  If $X$ is spin, it follows from Witten \cite{Witten} that the PMT is valid for $(\widetilde{X} \setminus \{ p\}, \widetilde{h})$, and the argument of Brendle-Chen shows that the inequality in (\ref{Aubin}) is strict (see the Appendix of \cite{EscobarJDG} for more details on reducing the PMT in the boundary case to the classical case by considering the double of the manifold).   Summarizing:

\begin{theorem} (See \cite{EscobarJDG}, \cite{BC}, \cite{Witten}) \label{YPThm}  The Yamabe invariant $Y(X,\partial X, [\overline{g}])$ is always attained by a smooth conformal metric satisfying (\ref{bdyYamabe}), provided one of the following holds:

\smallskip

$\mathrm{(i)}$ The dimension $3 \leq n \leq 5$.

$\mathrm{(ii)}$ The dimension $n \geq 6$, and $X$ is spin.
\end{theorem}

Using this result, we have

\begin{proposition}  \label{YPprop}  Let $(X, g_{+})$ be a Poincar\'e-Einstein manifold such that one of the following holds:

\smallskip

$\mathrm{(i)}$ The dimension $3 \leq n \leq 5$.


$\mathrm{(ii)}$ The dimension $n \geq 6$, and $X$ is spin.

\smallskip
\noindent
Assume that $(X,g_{+})$ is conformally compact of class $C^2$ and has a smooth representative in its conformal infinity.
Then there is a conformal compactification $\overline{g} = \rho^2 g_{+}$, at least $C^{3,\alpha}$ up to the boundary with $\alpha \in (0,1)$, satisfying

\smallskip
$\mathrm{(1)}$ The scalar curvature is constant:
\begin{align*}
R_{\overline{g}} = Y(X,\partial X,[\overline{g}]) \cdot Vol(X,\overline{g})^{-2/n}.
\end{align*}

$\mathrm{(2)}$ $\partial X$ is totally geodesic.
\end{proposition}

\begin{remark} The regularity statement follows from a result of Chru\'siel-Delay-Lee-Skinner \cite{CDLS}.   Previously, Anderson \cite{AndersonAM} used a compactification via the Yamabe problem with Dirichlet boundary conditions \cite{Ma} to study the regularity of Poincar\'e-Einstein metrics in four dimensions. \end{remark}

%
\section{Poincar\'e-Einstein metrics and Yamabe invariants}  \label{SecMainThm}
%

In this section we use the compactification of Proposition \ref{YPprop}
to prove a sharp inequality between the Yamabe invariant defined in (\ref{YXMdef}), and the Yamabe invariant of the boundary $M = \partial X$.  Let
$\gamma = \overline{g}\big|_{\partial X}$, and let $Y(M,[\gamma])$ denote the Yamabe invariant of the conformal infinity of $(X,g_{+})$:
\begin{align} \label{YMdef}
Y(M,[\gamma]) = \inf_{\tilde{\gamma} \in [\gamma]} \dfrac{ \int_M R_{\tilde{\gamma}} dV_{\tilde{\gamma}}}{Vol(M,\tilde{\gamma})^{(n-3)/(n-1)}}.
\end{align}
Also, we denote the isoperimetric ratio by $I(X,\partial X, \overline{g})$:
\begin{align*} 
I(X,\partial X, \overline{g}) \equiv \dfrac{ Vol(\partial X, \gamma)^n}{ Vol(X, \overline{g})^{n-1}}.
\end{align*}
An asymptotic expansion for solutions of singular Yamabe equations plays a crucial role in the proof of the following result.

\begin{theorem}  \label{MainThm1} Let $(X,g_{+})$ be a Poincar\'e-Einstein manifold satisfying the hypotheses of Proposition \ref{YPprop}, with Yamabe compactification $(X,\partial X, \overline{g}).$  Let $\gamma = \overline{g}|_{M}$ denote the induced metric, where $M = \partial X$.

If the dimension of $X$ is $n \geq 4$, then
\begin{align} \label{KeyIn}
Y(X,\partial X ,[\overline{g}]) \cdot I(X,\partial X, \overline{g})^{\frac{2}{n(n-1)}} \geq \frac{n}{n-2} Y(M,[\gamma]).
\end{align}
If the dimension is $n = 3$, then
\begin{align} \label{2d}
Y(X,\partial X,[\overline{g}]) \cdot I(X,\partial X, \overline{g})^{1/3} \geq 12 \pi \chi(M).
\end{align}
If equality occurs then $\overline{g}$ is Einstein and $\gamma$ has constant scalar curvature.
\end{theorem}

\begin{remark}  The inequality (\ref{KeyIn}) is sharp, in the sense that equality is achieved when $(X,M,\overline{g})$ is conformally equivalent to the hemisphere.
\end{remark}

\begin{proof}  The proof is based on the technique used by Obata to characterize the uniqueness of Yamabe metrics for Einstein manifolds \cite{Obata}.  In our setting, we have a boundary term, and this will require some additional information about the asymptotic behavior of solutions of the defining function.

To begin, let $\overline{g} = \rho^2 g_{+}$ be the Yamabe compactification of $g_{+}$ given by Proposition \ref{YPprop}.  By the standard formula which relates the trace-free Ricci tensor $E$ of $\overline{g}$ and $g_{+}$ we have
\begin{align} \label{Echange}  \begin{split}
E_{\overline{g}} &= E_{g_{+}} - (n-2) \rho^{-1} \big[ \nabla_{\overline{g}}^2 \rho - \frac{1}{n}  ( \Delta_{\overline{g}} \rho ) \overline{g} \big] \\
&=  - (n-2) \rho^{-1} \big[ \nabla_{\overline{g}}^2 \rho - \frac{1}{n}  ( \Delta_{\overline{g}} \rho ) \overline{g} \big],
\end{split}
\end{align}
where the second line follows from the Poincar\'e-Einstein condition.  Let $\epsilon > 0$ be small, and denote
\begin{align*}
X_{\epsilon} = \{ x \in X \ :\ d_{\overline{g}}(x,M) \geq \epsilon \},
\end{align*}
where $d_{\overline{g}}$ denotes the distance to the boundary with respect to $\overline{g}$.  We multiply both sides of (\ref{Echange}) by $\rho$, pair with $E_{\overline{g}}$, then integrate over $X_{\epsilon}$.  Since $E_{\overline{g}}$ is trace-free, we get
\begin{align*} 
\int_{X_{\epsilon}} | E_{\overline{g}}|^2_{\overline{g}} \rho  \hs dV_{\overline{g}} = - (n-2)  \int_{X_{\epsilon}} \overline{g}^{ik} \overline{g}^{j \ell} ( E_{\overline{g}})_{ij} ( \nabla_k \nabla_{\ell} \rho )  \hs dV_{\overline{g}}, \end{align*}
where the covariant derivatives are with respect to $\overline{g}$.  Next, we integrate by parts on the right-hand side: 
\begin{align} \label{s2}\begin{split}
\int_{X_{\epsilon}} | E_{\overline{g}}|^2_{\overline{g}} \rho\hs dV_{\overline{g}} &=   (n-2)  \int_{X_{\epsilon}} \overline{g}^{ik} \overline{g}^{j \ell} \nabla_k ( E_{\overline{g}})_{ij} \nabla_{\ell} \rho  \hs dV_{\overline{g}}\\
&\qquad - (n-2) \oint_{\partial X_{\epsilon}}\overline{g}^{j \ell} ( E_{\overline{g}} )_{ij} \nabla_{\ell} \rho  N^i \hs dS_{\overline{g}},
\end{split}\end{align}
where $N$ is the outward unit normal and $dS_{\overline{g}}$ is the area form on $\partial X_{\epsilon}$ with respect to $\overline{g}$.  By the contracted second Bianchi identity,
\begin{align*}
\overline{g}^{ik} \nabla_k ( E_{\overline{g}})_{ij} &= \frac{(n-2)}{2n} \nabla_j R_{\overline{g}} = 0,
\end{align*}
hence the interior term vanishes.

For the boundary term, we use (\ref{Echange}) once again to rewrite the trace-free Ricci term:
\begin{align} \label{s3} \begin{split}
- (n-2)& \oint_{\partial X_{\epsilon}}\overline{g}^{j \ell}  ( E_{\overline{g}} )_{ij} \nabla_{\ell} \rho  N^i \hs dS_{\overline{g}} \\
&= (n-2)^2 \oint_{\partial X_{\epsilon}}\overline{g}^{j \ell} \rho^{-1} \Big\{ \nabla_i \nabla_j \rho - \frac{1}{n}  ( \Delta_{\overline{g}} \rho) \overline{g}_{ij} \Big\} \nabla_{\ell} \rho  N^i \hs dS_{\overline{g}} \\
&= (n-2)^2 \oint_{\partial X_{\epsilon}} \rho^{-1} \Big\{ N^i \overline{g}^{j \ell}   \nabla_i \nabla_j \rho \nabla_{\ell} \rho - \frac{1}{n}  ( \Delta_{\overline{g}} \rho) (N^i \nabla_i \rho)  \Big\} \hs   dS_{\overline{g}} \\
&= \frac{(n-2)^2}{2} \oint_{\partial X_{\epsilon}} \rho^{-1} \Big\{  N \big( |\nabla \rho|^2_{\overline{g}} \big) - \frac{2}{n} ( \Delta_{\overline{g}} \rho) \big( N \rho \big)  \Big\} \hs   dS_{\overline{g}}.
\end{split}
\end{align}
Combining (\ref{s2}) and (\ref{s3}),
\begin{align} \label{s4}
\frac{2}{(n-2)^2} \int_{X_{\epsilon}} | E_{\overline{g}}|^2_{\overline{g}} \rho\hs dV_{\overline{g}} =  \oint_{\partial X_{\epsilon}} \rho^{-1} \Big\{ N \big( |\nabla \rho|^2_{\overline{g}} \big) - \frac{2}{n} ( \Delta_{\overline{g}} \rho) \big( N \rho \big)  \Big\} \hs   dS_{\overline{g}}.
\end{align}

To evaluate the boundary integral we first use the fact that $\overline{g} = \rho^2 g_{+}$ and $g_{+}$ has constant negative scalar curvature, which we have normalized to be $-n(n-1)$.  This implies (via the scalar curvature equation) that
\begin{align*} 
- \frac{2}{n} \Delta_{\overline{g}} \rho = - \dfrac{|\nabla_{\overline{g}} \rho|^2}{\rho} + \rho^{-1} + \frac{1}{n(n-1)}R_{\overline{g}} \rho.
\end{align*}
Therefore, we can rewrite (\ref{s4}) as
\begin{align} \label{s5} \begin{split}
&\frac{2}{(n-2)^2}  \int_{X_{\epsilon}}  | E_{\overline{g}}|^2_{\overline{g}} \rho\hs dV_{\overline{g}} \\
&\quad= \oint_{\partial X_{\epsilon}} \rho^{-1} \Big\{ N \big( |\nabla \rho|^2_{\overline{g}} \big) +    \big[ - \dfrac{|\nabla_{\overline{g}} \rho|^2}{\rho} + \rho^{-1} + \frac{1}{n(n-1)}R_{\overline{g}} \rho \big]( N \rho ) \Big\} \hs   dS_{\overline{g}}.
\end{split}
\end{align}

The second consequence of $g_{+}$ having constant negative scalar curvature is that it must be the unique solution of the Loewner-Nirenberg problem on $(X,\partial X, \overline{g})$ (see \cite{LN}, \cite{AM}).  That is, $g_{+}$ is the unique complete metric of constant negative scalar curvature defined in $X$ which is conformal to $\overline{g}$.
 Asymptotic expansions for solutions near the boundary have been carried out by various authors; e.g. see \cite{ACF1982CMP}.  To determine the boundary term in (\ref{s5}), we will need an expansion of the form
\begin{align} \label{formal}
\rho(x) = c_1 r + c_2 r^2 + c_3 r^3 + O(r^{3 + \alpha}),
\end{align}
where $r(x) = d_{\overline{g}}(x,M)$ denotes the distance to the boundary, and the coefficients are functions defined on $M$.  We will also need corresponding expansions for $|\nabla \rho|^2_{\overline{g}}$, $\partial_r\rho$, and $\partial_r|\nabla \rho|^2_{\overline{g}}$.

The existence of polyhomogeneous expansions and estimates for $\rho$ and its derivatives appear in \cite{ACF1982CMP} and \cite{Mazzeo1991}.  A formal expansion, with explicit expressions for the coefficients in (\ref{formal}), appears in \cite{GrahamSY} (see Section 4).  In our setting, since the boundary is totally geodesic we have
\begin{align} \label{cs} \begin{split}
c_1 &= 1, \\
c_2 &= 0, \\
c_3 &= - \frac{1}{3} \Big\{ \frac{1}{2(n-1)} R_{\overline{g}} - \frac{1}{2(n-2)} R_{\gamma} \Big\}.
\end{split}
\end{align}
Optimal estimates for the remainder in (\ref{formal}) under the assumption of $C^{3,\alpha}$-regularity can be modified from the estimates in \cite{HanJiang2014}.   In fact, we can derive these estimates by the maximum principle and scaled Schauder estimates, since these estimates involve only computable coefficients (i.e., the coefficients of the so-called local terms); see \cite{HanJiang2014}.

Using the formulas in (\ref{cs}), we can therefore write
\begin{align*}
\rho = r - \frac{1}{3} A r^3 + O(r^{3+\alpha}),
\end{align*}
where
\begin{align} \label{A}
A = \frac{1}{2(n-1)} R_{\overline{g}} - \frac{1}{2(n-2)} R_{\gamma}.
\end{align}
In addition,
\begin{align*}
\rho^{-1} &= r^{-1} + \frac{1}{3} A r + O(r^{1+\alpha}), \\
|\nabla \rho|^2_{\overline{g}} &= 1 - 2 A r^2 + O(r^{2+\alpha}).
\end{align*}
Using the fact that $N = - \frac{\partial}{\partial r}$, it follows that on $\partial X_{\epsilon}$,
\begin{align} \label{bdy}  \begin{split}
\rho^{-1} \Big\{ N \big( |\nabla \rho|^2_{\overline{g}} \big) +    \big[ - \dfrac{|\nabla_{\overline{g}} \rho|^2}{\rho} + & \rho^{-1} +  \frac{1}{n(n-1)}R_{\overline{g}} \rho \big]( N \rho ) \Big\} \\
& = 2 A - \frac{1}{n(n-1)} R_{\overline{g}} + O(\epsilon^\alpha).
\end{split}
\end{align}
By (\ref{A}),
\begin{align}\label{eq-relation}
2 A - \frac{1}{n(n-1)} R_{\overline{g}} = \frac{1}{n} \big[ R_{\overline{g}} - \frac{n}{n-2}  R_{\gamma} \big].
\end{align}
Substituting (\ref{eq-relation}) into (\ref{bdy}) gives
\begin{align} \label{limit} \begin{split}
\lim_{\epsilon \to 0} \oint_{\partial X_{\epsilon}} \rho^{-1} \Big\{ N \big( |\nabla \rho|^2_{\overline{g}} \big) +    \big[ - \dfrac{|\nabla_{\overline{g}} \rho|^2}{\rho}  & + \rho^{-1} + \frac{1}{n(n-1)}R_{\overline{g}} \rho \big]( N \rho ) \Big\} \hs   dS_{\overline{g}}  \\
 & = \frac{1}{n} \oint_{M} \big[ R_{\overline{g}} -  \frac{n}{n-2} R_{\gamma} \big] dV_{\gamma}.
\end{split}
\end{align}
From (\ref{s5}) we conclude
\begin{align} \label{s55}
\frac{1}{n} \oint_{M} \big[ R_{\overline{g}} -  \frac{n}{n-2} R_{\gamma} \big] dV_{\gamma} = \frac{2}{(n-2)^2}  \int_{X} | E_{\overline{g}}|^2_{\overline{g}} \rho\hs dV_{\overline{g}}.
\end{align}
This implies, by Proposition \ref{YPprop} and the definition (\ref{YMdef}),
\begin{align} \label{s6}  \begin{split}
Y(X,&\partial X,[\overline{g}]) \hs  Vol(X,\overline{g})^{-2/n} \hs Vol(M,\gamma) \\
 &= \oint R_{\overline{g}} \hs dV_{\gamma}  \\
&= \frac{n}{n-2} \oint_M R_{\gamma} dV_{\gamma} + \frac{2n}{(n-2)^2}  \int_{X} | E_{\overline{g}}|^2_{\overline{g}} \rho\hs dV_{\overline{g}} \\
&\geq \frac{n}{n-2} Y(M,[\gamma]) \hs Vol(M,\gamma)^{(n-3)/(n-1)} + \frac{2n}{(n-2)^2}  \int_{X} | E_{\overline{g}}|^2_{\overline{g}} \rho\hs dV_{\overline{g}}.
\end{split}
\end{align}
Dropping the integral over $X$ and dividing by the volume of $M$ we get
\begin{align*}
Y(X,\partial X,[\overline{g}]) \hs Vol(X,\overline{g})^{-2/n}  \geq \frac{n}{n-2} Y(M,[\gamma]) \hs Vol(M,\gamma)^{-2/(n-1)},
\end{align*}
which implies (\ref{KeyIn}).  In addition, if equality holds in (\ref{KeyIn}) then $E_{\overline{g}} \equiv 0$, hence $\overline{g}$ is Einstein.  Since $M$ is totally geodesic, the Gauss curvature equation implies
\begin{align*}
R_{\overline{g}} &= 2 Ric_{\overline{g}}(N,N) + R_{\gamma} \\
&= \frac{2}{n} R_{\overline{g}} + R_{\gamma},
\end{align*}
and $R_{\gamma}$ must be constant.
\end{proof}

An immediate corollary of Theorem \ref{MainThm1} is the following result first proved by J. Qing \cite{JQ}:

\begin{corollary} (See \cite{JQ}; also \cite{CQY}) Let $(X,g_{+})$ be a Poincar\'e-Einstein manifold satisfying the hypotheses of Proposition \ref{YPprop}, and let $(M,[\gamma])$ denote its conformal infinity. If
\begin{align*} 
Y(M,[\gamma]) > 0,
\end{align*}
then the Yamabe invariant (\ref{YXMdef}) must be positive:
\begin{align*}
Y(X,\partial X, [\bar{g}]) > 0.
\end{align*}
\end{corollary}

In fact, J. Qing proved the existence of a defining function $v$ such that the metric $v^2 g_{+}$ has totally geodesic boundary, and the scalar curvature satisfies
\begin{align*}
R_{v^2 g_{+}} \geq \frac{n}{n-2} Y(M, [\gamma]) Vol(M,\overline{\gamma})^{2/(n-1)},
\end{align*}
where $\overline{\gamma} \in [\gamma]$ is a Yamabe metric in the conformal infinity.  Integrating this over $X$ we obtain and inequality that is weaker than (\ref{KeyIn}).  On the
other hand, both results rely on the solution of the Yamabe problem (either for the boundary or interior).

\section{Obstructions to Poincar\'e-Einstein fillings} \label{CorSec}

In this section we prove Theorem \ref{MainIntro}.  Since the proof uses in a crucial way the construction by Gromov-Lawson \cite{GL} of metrics with PSC on the seven-dimensional sphere $S^7$, we will provide a brief sketch (see \cite{Rosenberg} for a nice survey with related results).

The starting point is an earlier construction of PSC metrics (e.g., \cite{GL2}) on the total space of an oriented $\mathbb{R}^4$-bundle $\mathcal{E}$ over $S^4$.  This gives a metric of PSC on the unit-disk bundle $D(\mathcal{E})$; furthermore, the metric can be made a product near the boundary; i.e., the unit sphere bundle $\Sigma(\mathcal{E})$.  A result of Milnor \cite{M1} provides criteria in terms of the Euler number and Pontrjagin number of $\mathcal{E}$ for determining when $\Sigma(\mathcal{E})$ is diffeormorphic to the standard $S^7$.  Using this result, one can construct a sequence of metrics (by varying the bundle $\mathcal{E}$) of PSC metrics on $S^7$.  By a relative index calculation, Gromov-Lawson showed that these metrics are in different components of $\mathfrak{R}^{+}(S^7)$ (see Section 4 of \cite{GL}).

The relevant point for us is that these metrics cannot be extended to metrics of PSC on $B^8$: if they could, then $N^8 = D(\mathcal{E}) \cup B^8$ would admit a metric of positive scalar curvature.  Since $N^8$ is spin, it would follow that the $\widehat{A}$-genus of $N^8$ vanishes, but using another result of Milnor (\cite{M2}) one can compute the $\widehat{A}$-genus explicitly in terms of the Pontrjagin number of $\mathcal{E}$ and see that it is non-zero.

In the following, let $Y^8 = D(\mathcal{E})$ denote the unit disk bundle with $\partial Y^8 = \Sigma(\mathcal{E}) \approx S^7$, and let $\eta$ be a metric of positive scalar curvature on $Y^8$ as described above.  The induced metric on $S^7 = \partial Y^8$ will be denoted by $\eta_0$.

\begin{theorem} \label{Cor2} $(S^7, [\eta_0])$ cannot be the conformal infinity of a Poincar\'e-Einstein metric on the ball $B^8$. \end{theorem}

\begin{proof}  Suppose $(B^8, g_{+})$ is a Poincar\'e-Einstein metric whose conformal infinity is $[\eta_0]$.  Let $\overline{g}$ denote the Yamabe compactification of $g_{+}$ given by Proposition \ref{YPprop} (with slight modifications to the proof we could also use the defining function constructed by J. Qing in \cite{JQ}).  Note that regularity will not be a consideration, since by \cite{CDLS} there is a compactification which is smooth up to the boundary.

Since $Y(S^7,[\eta_0]) > 0$, it follows that $R_{\overline{g}} > 0$.  This is not an immediate contradiction: the Yamabe metric is an extension of a metric in the conformal class of $\eta_0$, but not necessarily of $\eta_0$.  However, we will use the fact that the entire construction is, in some sense, conformally invariant.  To this end, since $\overline{g}|_{S^7}$ is conformal to $\eta_0$ we can write
\begin{align} \label{vdef}
\eta_0 = v^{4/(n-2)}\overline{g} \big|_{S^7} = v^{2/3}\overline{g} \big|_{S^7}
\end{align}
for some function $v > 0$ on $S^7$.  As a first step we want to extend $\eta_0$ inside $B^8$ to a metric which is conformal to $\overline{g}$.  To this end, for $\delta_0 > 0$ small let
\begin{align*} 
V_0 = \{ p \in B^8 \ |\ d_{\overline{g}}(p,S^7) < \delta_0  \}
\end{align*}
denote a collar neighborhood of the boundary of $B^8$.  We fix $\delta_0 > 0$ small enough so that $V_{0}$ can be identified with $\partial B^8 \times [0,\delta_0)$ via the normal exponential map; i.e., given $p \in V_0$ we can write $p = (x,\tau)$ to mean $p$ is obtained by following the unit speed geodesic starting at $x \in \partial B^8$ with initial velocity given by the inward unit normal for time $\tau$.  In $V_{0}$, we define the function $v_1 : V_0 \cong S^7 \times [0,\delta_0) \rightarrow \mathbb{R}^{+}$ by
\begin{align*} 
v_1(x,\tau) = v(x),
\end{align*}
where $v$ is defined by (\ref{vdef}).  By construction,
\begin{align} \label{Nbv}
N_{\overline{g}}v_1 \big|_{S^7} = 0,
\end{align}
where $N_{\overline{g}}$ is the outward normal with respect to $\overline{g}$.  Let $\chi_0$ be a cut-off function with $0 \leq \chi_0 \leq 1$, $\chi_0 \equiv 0$ near $\partial B^8$, and $\chi_0 \equiv 1$ on $B^8 \setminus V_0$.
We then define the conformal factor $v : \overline{B}^8 \rightarrow \mathbb{R}^{+}$ by
\begin{align*} 
v = \chi_0 + ( 1 - \chi_0) v_1,
\end{align*}
and the conformal metric
\begin{align*} 
\widetilde{g} = v^{2/3} \overline{g}.
\end{align*}
By construction,
\begin{align*}
\widetilde{g} \big|_{S^7} = v^{2/3}\overline{g}\big|_{S^7} = \eta_0.
\end{align*}

Next, we want to show that $S^7 = \partial B^8$ is totally geodesic with respect $\widetilde{g}$.  To see this, we recall the formula for the transformation of the mean curvature under a conformal change of metric: if $H_{\overline{g}}$ and $H_{\widetilde{g}}$ are the mean curvatures of $\partial B^8$ with respect to $\overline{g}$ and $\widetilde{g}$, then in dimension $n = 8$,
\begin{align} \label{Hcon}
N_{\overline{g}}v + 3 H_{\overline{g}} v = 3 H_{\widetilde{g}} v^{4/3},
\end{align}
where $N_{\overline{g}}$ is the outward normal with respect to $\overline{g}$.  Recall by Proposition \ref{YPprop} that $H_{\overline{g}} = 0$.  Also, since $v = v_1$ near the boundary,
it follows from (\ref{Nbv}) that
\begin{align*} 
N_{\overline{g}}v \big|_{\partial B^8} = 0.
\end{align*}
From (\ref{Hcon}) we see that $H_{\widetilde{g}} = 0$.  Since the boundary is totally
umbilic with respect to $\overline{g}$, and this condition is conformally invariant, it must be totally umbilic with
respect to $\widetilde{g}$.  Since the mean curvature is zero, it follows that the boundary is totally geodesic.

Recall $\eta$ is a product metric near the boundary of $Y^8$, so we can write
\begin{align} \label{split}
\eta = ds^2 + \eta_0
\end{align}
where $s \in [0,\epsilon_0)$, with $\epsilon_0 > 0$ small.   In particular $\partial Y^8 = S^7$ is totally geodesic in $Y^8$. Since $(Y^8,\eta)$ and $(B^8,\widetilde{g})$ have the same induced metric on their common boundary, and since the boundary is totally geodesic with respect to both metrics,
it follows that the metric
\begin{align} \label{g0}
g_0 = \begin{cases} \eta \ \ \mbox{ on } Y^8, \\
\widetilde{g} \ \ \mbox{ on }B^8,
\end{cases}
\end{align}
is $C^1$ on the closed manifold $N^8 = Y^8 \cup B^8$.  We want to argue that this implies that $N^8$ admits a conformal metric of positive scalar curvature, which, as we observed above, is a contradiction.  To make the argument work, however, we will need to modify $\wtg$ near the boundary of $B^8$ in order to construct a $C^2$-metric on $N^8$.  We now proceed to do this, and then explain how it can be used to construct a PSC metric.

For $\delta > 0$ small, let
\begin{align*} 
\Ud = \{ p \in B^8\ :\ d_{\wtg}(p,\partial B^8) < \delta \}
\end{align*}
be a collar neighborhood of the boundary of $B^8$, where $d_{\wtg}$ is the distance with respect to $\wtg$.  For $\delta > 0$ sufficiently small, we can express $\wtg$ in $\Ud$ as
\begin{align*} 
\wtg = dt^2 + \eta_0 + t^2 h + k,
\end{align*}
where
\begin{align*}  
t &= d_{\wtg}(\cdot, \partial B^8), \\
h &= Rm_{\wtg}( \cdot, \tN, \cdot, \tN),
\end{align*}
with $Rm_{\wtg}$ the curvature tensor and $\tN$ the outward unit normal on $\partial B^8$ with respect to $\wtg$, and $k$ is a tensor satisfying
\begin{align*}
k = O(t^3)
\end{align*}
(see \cite{PS}, Section 5).  Let $\psi \in C^{\infty}(\mathbf{R})$ be a cut-off function with $0 \leq \psi \leq 1$,
\begin{align*}
\psi(t) =
\begin{cases}
0, \ \mbox{for } t \leq \delta/2, \\
1, \ \mbox{for } t \geq \delta.
\end{cases}
\end{align*}
We also assume
\begin{align} \label{Dpsi}
|\psi'(t)| \leq \dfrac{C}{\delta}, \ \ |\psi''(t)| \leq \dfrac{C}{\delta^2}.
\end{align}
Now define
\begin{align}  \label{gddef}
\wtd = \begin{cases}
dt^2 + \eta_0 + \psi(t) \big( t^2 h + k \big) \ \ \mbox{in } U_{\delta}, \\
\wtg \ \ \mbox{in } B^8 \setminus U_{\delta}.
\end{cases}
\end{align}
Notice in $U_{\delta/2}$,
\begin{align*} 
\wtd = dt^2 + \eta_0,
\end{align*}
so that $\wtd$ is a product metric in a small neighborhood of $\partial B^8$.  By (\ref{split}), we can identify $\eta$ and $\wtd$ in a neighborhood of $S^7$ to define a smooth metric on $N^8 = Y^8 \cup B^8$, which we will also denote by $\wtd$.

It follows from (\ref{gddef}) and (\ref{Dpsi}) that the second derivatives of $\wtd$ are bounded, independent of $\delta$.  Therefore, we can take a subsequence $\delta_i \rightarrow 0$ and the metrics $\wti = \widetilde{g}_{\delta_i}$ will converge in $C^{1,\alpha}$, for some fixed $\alpha \in (0,1)$, to the $C^1$-metric $g_0$ in (\ref{g0}).  Also, if $\Ri$ denotes the scalar curvature with respect to $\wti$ then
\begin{align} \label{Rd2}
| \Ri | \leq C
\end{align}
for some $C$ (independent of $i$).  Of course, we also have
\begin{align*} 
\Ri = \begin{cases} R_{\eta} \ \ \mbox{ on } Y^8, \\
R_{\wtg} \ \ \mbox{ on } B^8 \setminus U_{\delta_i}.
\end{cases}
\end{align*}

Let
\begin{align} \label{Li}
L_i = - \frac{14}{3} \Delta_{\wti} + \Ri
\end{align}
denote the conformal laplacian on $(N^8, \wti)$.  Let $\lambda_i$ denote the principal eigenvalue of $L_i$ and $u_i > 0$ the first eigenfunction, normalized to have unit $L^2$-norm:
\begin{align} \label{uidef} \begin{split}
L_i u_i &= \lambda_i u_i, \\
\int_{N^8} u_i^2 \ & \di = 1.
\end{split}
\end{align}
If we let $u_i^0$ be the constant function normalized so that
\begin{align*}
\int_{N^8} (u_i^0)^2 \di = 1,
\end{align*}
then for some constant $b_0 > 0$ we have
\begin{align*}
b_0^{-1} \leq u_i^0 \leq b_0.
\end{align*}
Also,
\begin{align*}
\lambda_i \leq \int_{N^8} u_i^0 L_i u_i^0 \ \di = (u_i^0)^2 \int_{N^8} \Ri \di \leq C,
\end{align*}
hence the sequence $\{ \lambda_i \}$ is bounded above.

Next, we will prove that $\{\lambda_i \}$ has a positive, uniform lower bound for $i$ large.  If we write the eigenvalue equation for $u_i$ in local coordinates, then $u_i$ satisfies a second order elliptic equation of the form
\begin{align*}
a^{k \ell} \partial_k \partial_{\ell} u_i + b^k \partial_k u_i + c u_i = 0.
\end{align*}
Since $\{ \wti \}$ converges in $C^{1,\alpha}$ and the curvature of $\wti$ is uniformly bounded, it follows that $\{ a^{k  \ell}\}$ is uniformly elliptic, the coefficients $a^{k \ell}$ and $b^k$ are bounded in $C^{\alpha}$, and $c$ is bounded in $L^{\infty}$.  By standard elliptic estimates $\{ u_i \}$ is bounded in $W^{2,p}$, for any $p >> 1$, with respect to some fixed background metric.  We can therefore take a subsequence (still denoted by $\{ u_i \}$) which converges in $C^{1,\gamma}$, for some $\gamma \in (0,1)$.   By decreasing $\alpha$ or $\gamma$ if necessary, we may assume $\alpha = \gamma$.

By (\ref{uidef}), we have
\begin{align*}
\lambda_i = \int_{N^8} u_i L_i u_i \di
= \frac{14}{3} \int_{N^8} |\nabla_{\wti} u_i |^2 \di +  \int_{N^8} \Ri u_i^2 \di.\end{align*}
We write
\begin{align} \label{ev1} \lambda_i=I + II,\end{align}
where
\begin{align*}
I&=\frac{14}{3} \int_{Y^8} |\nabla_{\wti} u_i |^2 \di +  \int_{Y^8} \Ri u_i^2 \di,\\
II&=\frac{14}{3} \int_{B^8} |\nabla_{\wti} u_i |^2 \di +  \int_{B^8} \Ri u_i^2 \di.
\end{align*}
We first estimate $I$. Since $\widetilde{g}_i=\eta$ and
$R_{\eta} \geq \rho_0 > 0$ in $Y^8$, we can easily estimate
\begin{align} \label{I}
I =\frac{14}{3} \int_{Y^8} |\nabla_{\eta} u_i |^2 dV_{\eta} +  \int_{Y^8} R_{\eta} u_i^2 dV_{\eta}
\geq \rho_0 \int_{Y^8}  u_i^2 dV_{\eta}.
\end{align}
To estimate $II$, we will split the integral over $B^8$ into two parts: an integral over a collar neighborhood of the boundary, and an integral over the complement.  The key point is that on the former set, the integrals will be small while on the latter set $\wti$ is conformal to $\overline{g}$.

It will simplify our estimates if we define the collar neighborhoods in terms of $\overline{g}$.  Since $\widetilde{g} = v^{2/3}\overline{g}$ on $B^8$, it follows that distances measured between points in $\overline{B^8}$ with respect to $\widetilde{g}$ and $\overline{g}$ are comparable.  In particular, if $p,q \in \overline{B^8}$, then
\begin{align*} 
c_1^{-1} d_{\overline{g}}(p,q) \leq d_{\widetilde{g}}(p,q) \leq c_1 d_{\overline{g}}(p,q),
\end{align*}
for some $c_1 > 0$.  Therefore, if we define the collar neighborhoods
\begin{align*} 
V_i = \{ p \in B^8\ :\ d_{\overline{g}}(p,\partial B^8) < c_1 \delta_i \},
\end{align*}
then
\begin{align*}  
U_{\delta_i} \subset V_i.
\end{align*}
In particular, on $B^8 \setminus V_i$, $\wti = \widetilde{g}$.

Returning to our estimate of $II$ in (\ref{ev1}), we write
\begin{align} \label{ev2}
II = II_1 + II_2,
\end{align}
where
\begin{align*}II_1&=\frac{14}{3} \int_{V_i} |\nabla_{\wti} u_i |^2 \di +  \int_{V_i} \Ri u_i^2 \di,  \\
II_2 &= \frac{14}{3} \int_{B^8 \setminus V_i} |\nabla_{\wti} u_i |^2 \di +  \int_{B^8 \setminus V_i} \Ri u_i^2 \di.
\end{align*}
By (\ref{Rd2}), the fact that $\{ \wti \}$ and $\{ u_i \}$ converge in $C^{1,\alpha}$, and
\begin{align} \label{volU}
Vol_{\wti}(V_i) \leq C \delta_i,
\end{align}
we have
\begin{align} \label{ev3} \begin{split}
II_1 &= \frac{14}{3} \int_{V_i} |\nabla_{\wti} u_i |^2 \di +  \int_{V_i} \Ri u_i^2 \di \\
&\geq  \int_{V_i} \Ri u_i^2 \di \\
&\geq - C \int_{V_i} u_i^2 \di \\
&\geq - C \delta_i.
\end{split}
\end{align}
To estimate $II_2$, we use the fact we observed above when defining $V_i$; i.e., on $B^8 \setminus V_i$, $\wti = \wtg$.
Therefore,
\begin{align} \label{ev4}
II_2 &= \frac{14}{3} \int_{B^8 \setminus V_i} |\nabla_{\wtg} u_i |^2 dV_{\wtg} +  \int_{B^8 \setminus V_i} R_{\wtg} u_i^2 dV_{\wtg}.
\end{align}
Since $\wtg = v^{4/(n-2)} \overline{g}$,
\begin{align*}
R_{\wtg} &= v^{-5/3} \big\{ -\frac{14}{3} \Delta_{\overline{g}} v + R_{\overline{g}} v \big\}, \\
dV_{\wtg} &= v^{8/3}dV_{\overline{g}}, \\
|\nabla_{\wtg} u_i|^2 &= |\nabla_{\overline{g}} u_i|^2 v^{-2/3}.
\end{align*}
Substituting these into (\ref{ev4}) and integrating by parts gives
\begin{align*} 
II_2 &= \frac{14}{3} \int_{B^8 \setminus V_i} |\nabla_{\overline{g}} u_i |^2 v^2 dV_{\overline{g}}  +  \int_{B^8 \setminus V_i} \big\{ -\frac{14}{3} v \Delta_{\overline{g}} v + R_{\overline{g}} v^2 \big\} u_i^2 dV_{\overline{g}} \\
&= \frac{14}{3} \int_{B^8 \setminus V_i} |\nabla_{\overline{g}} u_i |^2 v^2 dV_{\overline{g}} +  \frac{14}{3} \int_{B^8 \setminus V_i} |\nabla_{\overline{g}} v|^2 u_i^2 dV_{\overline{g}} \\
 & \hskip.25in + 2 \big( \frac{14}{3}\big) \int_{B^8 \setminus V_i} \langle \nabla_{\overline{g}} v, \nabla_{\overline{g}} u_i \rangle_{\overline{g}} v u_i  dV_{\overline{g}} + \int_{B^8 \setminus V_i} R_{\overline{g}} v^2  u_i^2 dV_{\overline{g}} \\
 & \hskip.5in - \frac{14}{3} \oint_{\Sigma_i } u_i^2 v ( N_{\overline{g}}v)  dS_{\overline{g}} \\
 &= \frac{14}{3} \int_{B^8 \setminus V_i} |\nabla_{\overline{g}} (u_i v) |^2 dV_{\overline{g}}  + \int_{B^8 \setminus V_i} R_{\overline{g}} v^2  u_i^2 dV_{\overline{g}} \\
 & \hskip.5in - \frac{14}{3} \oint_{\Sigma_i } u_i^2 v ( N_{\overline{g}}v)  dS_{\overline{g}} \\
 &\geq \int_{B^8 \setminus V_i} R_{\overline{g}} v^2  u_i^2 dV_{\overline{g}} - \frac{14}{3} \oint_{\Sigma_i } u_i^2 v ( N_{\overline{g}}v)  dS_{\overline{g}},
\end{align*}
where $\Sigma_i = \partial (B^8 \setminus V_i)$, $N_{\overline{g}}$ is the unit outward normal, and $dS_{\overline{g}}$ the boundary measure on $\Sigma_i$ with respect to $\overline{g}$.
Recall that near $\partial B^8$, $v$ is defined to be constant on normal geodesics.  Since $\Sigma_i$ is the set of points of distance $c_1 \delta_i$ from the boundary, it follows that $N_{\overline{g}} v = 0$ on $\Sigma_i$.  Therefore, the boundary integral vanishes, and
\begin{align*} 
II_2 \geq  \int_{B^8 \setminus V_i} R_{\overline{g}} v^2  u_i^2 dV_{\overline{g}}.
\end{align*}
Using the facts that $v \geq v_0 > 0$ on $B^8$ for some constant $v_0$, $R_{\overline{g}}$ is constant, and $u_i$ converges in $C^{1,\alpha}$, we conclude
\begin{align} \label{ev7-2} \begin{split}
II_2 &\geq  R_{\overline{g}} v_0^2\int_{B^8 \setminus V_i}  u_i^2 dV_{\overline{g}} \\
&= R_{\overline{g}} v_0^2\int_{B^8} u_i^2 dV_{\overline{g}} - R_{\overline{g}} v_0^2\int_{V_i} u_i^2 dV_{\overline{g}} \\
&\geq R_{\overline{g}} v_0^2\int_{B^8} u_i^2 dV_{\overline{g}} - C \delta_i,
\end{split}
\end{align}
where the last line follows from (\ref{volU}). 
Combining 
(\ref{ev3}) and (\ref{ev7-2}), we obtain
\begin{align} \label{II}
II \geq  R_{\overline{g}} v_0^2\int_{B^8} u_i^2 dV_{\overline{g}}  - C \delta_i.
\end{align}

It follows from (\ref{I}) and (\ref{II}) that
\begin{align*}
\lambda_i \geq \rho_0 \int_{Y^8}  u_i^2 dV_{\eta} + R_{\overline{g}} v_0^2\int_{B^8} u_i^2 dV_{\overline{g}}  - C \delta_i.
\end{align*}
Since $\{ u_i \}$ converges in $C^{1,\alpha}$ to a non-zero function, we see that $\lambda_i$ has a positive, uniform lower bound for large $i$.  In particular, this implies that $\wti$ is conformal to a metric of positive scalar curvature, a contradiction.
\end{proof}

\begin{remark}  There is a construction of PSC metrics on spheres in dimensions $4k -1$, for all $k \geq 2$, which cannot be extended to metrics of PSC in the ball (see Section 2 of \cite{Rosenberg} for an outline).  Therefore, a result analogous to Theorem \ref{Cor2} should hold for spheres of these dimensions as well.  \end{remark}

%
%
%

\end{document}